\documentclass[leqno,11pt,a4paper]{amsart}
\usepackage{graphicx}
\usepackage[usenames,dvipsnames]{color}
\usepackage{hyperref}
\usepackage{array, ragged2e}
\newtheorem{theorem}{Theorem}[section]
\newtheorem{corollary}[theorem]{Corollary}
\newtheorem{lemma}[theorem]{Lemma}
\newtheorem{proposition}[theorem]{Proposition}
\theoremstyle{definition}

\newtheorem{example}[theorem]{Example}
\newtheorem{remark}[theorem]{Remark}
\newtheorem{algorithm}[theorem]{Algorithm}
\numberwithin{equation}{section}
\newenvironment{acknowledge}{\noindent\textbf{Acknowledgments.}}{}
\renewcommand{\min}[1]{\mathrm{min}\!\left\{#1\right\}}
\renewcommand{\max}[1]{\mathrm{max}\!\left\{#1\right\}}
\newcommand{\lcm}[1]{\mathrm{lcm}\!\left\{#1\right\}}
\renewcommand{\gcd}[1]{\mathrm{gcd}\!\left\{#1\right\}}
\newcommand{\sconv}[1]{\mathrm{conv}\!\left\{#1\right\}}
\newcommand{\Vol}[1]{\mathrm{Vol}\!\left(#1\right)}
\newcommand{\pos}[1]{\mathrm{pos}\!\left(#1\right)}
\newcommand{\sspan}[1]{\mathrm{span}\!\left\{#1\right\}}
\newcommand{\V}[1]{\mathcal{V}\!\left(#1\right)}
\newcommand{\F}[1]{\mathcal{F}\!\left(#1\right)}
\newcommand{\0}{\mathbf{0}}
\newcommand{\idx}{\ell}
\newcommand{\locidx}{l}
\newcommand{\maxlocidx}{m}
\newcommand{\order}{o}
\newcommand{\intr}[1]{{#1}^\circ}
\newcommand{\Z}{\mathbb{Z}}
\newcommand{\Q}{\mathbb{Q}}
\newcommand{\R}{\mathbb{R}}
\newcommand{\rand}{\partial}
\newcommand{\card}[1]{{\mid\! #1 \!\mid}}
\newcommand{\pro}[2]{\left \langle {\eta}_{#1}, #2 \right \rangle}
\newcommand{\spro}[2]{\left \langle #1, #2 \right \rangle}
\newcommand{\Hom}[1]{\mathrm{Hom}\!\left(#1\right)}
\newcommand{\st}{\,:\,}
\graphicspath{{images/}}
\begin{document}
\author[A.~M.~Kasprzyk]{Alexander M.~Kasprzyk}
\address{IMSAS, University of Kent, Canterbury, CT2 7NF, United Kingdom}
\email{a.m.kasprzyk@kent.ac.uk}
\author[M.~Kreuzer]{Maximilian Kreuzer}
\address{Institute for Theoretical Physics, Vienna University of Technology, Wiedner Hauptstrasse~8-10, 1040 Vienna, Austria}
\email{maximilian.kreuzer@tuwien.ac.at}
\author[B.~Nill]{Benjamin Nill}
\address{Freie Universit\"{a}t Berlin, Institut f\"{u}r Mathematik, Arbeitsgruppe Gitterpolytope, Arnimallee~3, 14195 Berlin, Germany}
\email{nill@math.fu-berlin.de}
\subjclass[2000]{52B20 (Primary); 14M25, 14Q10 (Secondary)}
\title[Combinatorial classification of toric log del Pezzo surfaces]{On the combinatorial classification of\\ toric log del Pezzo surfaces}
\begin{abstract}
Toric log del Pezzo surfaces correspond to convex lattice polygons containing the origin in their interior and having only primitive vertices. An upper bound on the volume and on the number of boundary lattice points of these polygons is derived in terms of the index $\idx$. Techniques for classifying these polygons are also described: a direct classification for index two is given, and a classification for all $\idx\leq16$ is obtained.
\end{abstract}
\maketitle
\section{Introduction}\label{sec:introduction}
Motivated by the algebro-geometric question of classifying toric log del Pezzo surfaces we investigate, from a purely combinatorial viewpoint, lattice polygons containing the origin in their interior.

A normal complex surface $X$ is called a \emph{log del Pezzo surface} if it has at worst log terminal singularities and if its anticanonical divisor $-K_X$ is a $\Q$-Cartier ample divisor. The smallest positive multiple $\idx$ for which $-\idx K_X$ is Cartier is called the \emph{index} of $X$. Such surfaces have been studied extensively: for example by Nukulin~\cite{Nik89a,Nik88,Nik89}, Alexeev and Nukulin~\cite{AN06}, and Nakayama~\cite{Nak07}. There has also been considerable emphasis on classification results in the rank one case (i.e. when the Picard number is one): see~\cite{Ye02,Koj03}.

If, in addition to being a log del Pezzo surface, $X$ is also toric (i.e. $X$ contains an algebraic torus as a dense open subset, together with an action of the torus on $X$ which extends the natural action of the torus on itself) then we call $X$ a \emph{toric log del Pezzo surface}. There exists a bijective correspondence between toric log del Pezzo surfaces and certain convex lattice polygons: the \emph{LDP-polygons}.

Fix a lattice $N\cong\Z^2$ and let $Q\subset N_\Q:=N\otimes_\Z\Q$ be a lattice polygon; i.e. $Q$ is the convex hull of finitely many lattice points, and has non-empty interior. We denote the vertices of $Q$ by $\V{Q}$ and the facets (also called edges) by $\F{Q}$. By the volume $\Vol{Q}$ we mean the normalised volume, which equals twice the Euclidean volume. By $\rand Q$ we mean the boundary of $Q$.

\begin{itemize}
\item $Q$ is called an \emph{IP-polygon} if it contains the origin in its (strict) interior; we write $\0\in\intr{Q}$.
\item An IP-polygon $Q$ is called an \emph{LDP-polygon} if the vertices of $Q$ are primitive lattice points, i.e. if no lattice point lies strictly between the origin and a vertex.
\end{itemize}

Let $Q$ be an LDP-polygon and let $X_Q$ be the toric surface whose fan is generated by the faces of $Q$. Then $X_Q$ is a log del Pezzo surface. Furthermore any toric log del Pezzo surface can be derived in this fashion. Two toric log del Pezzo surfaces are isomorphic if and only if the corresponding polygons are unimoduar equivalent. The toric log del Pezzo surface has rank one if and only if the polygon is a triangle. For further details on toric varieties consult~\cite{Oda78,Ful93}. For more information about LDP-polygons see~\cite[\S6]{Dais06}, \cite[\S1]{Dais07} and~\cite[\S2]{DN08}.

Let $\spro{\cdot}{\cdot}$ be the pairing between the lattice $N$ and its dual $M:=\Hom{N,\Z}$. Let $F\in\F{Q}$ be a facet of $Q$. The unique primitive lattice point in the dual lattice defining an outer normal of 
$F$ is denoted by $\eta_F$. The integer $\locidx_F:=\pro{F}{F}$ equals the integral distance between $\0$ and $F$, and is called the \emph{local index} of $F$ (with respect to $Q$).

We now define three important invariants of $Q$:
\begin{itemize}
\item The \emph{order} is given by $\order_Q:=\min{k\in\Z_{\geq 1}\st\intr{(Q/k)}\cap N=\{\0\}}$;
\item The \emph{maximal local index} is given by $\maxlocidx_Q:=\max{\locidx_F\st F\in\F{Q}}$;
\item The \emph{index} is given by $\idx_Q := \lcm{\locidx_F\st F \in \F{Q}}$.
\end{itemize}
Amongst these invariants is the following hierarchy:
\begin{equation}
\label{eq:hierarchy}
\order_Q\leq\maxlocidx_Q\leq\idx_Q.
\end{equation}
Figure~\ref{fig:strict_inequalities} gives an example of an LDP-polygon $Q$ for which the inequalities are strict.

\begin{figure}[htbp]
\centering
\includegraphics[scale=0.8]{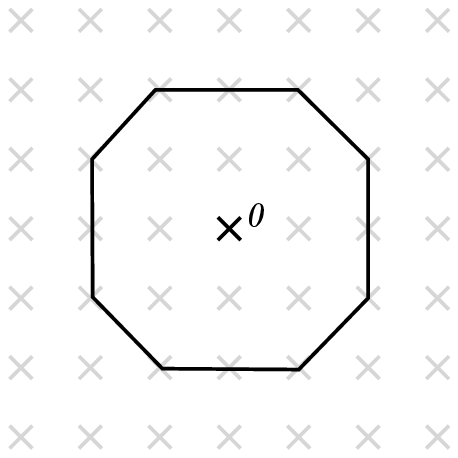}
\caption{An example of an LDP-polygon with $\order_Q=2$, $\maxlocidx_Q=3$, and $\idx_Q=6$.}
\label{fig:strict_inequalities}
\end{figure}

Whilst the second inequality is trivial, let us explain the first. Let $k=\order_Q$. Then there exists a lattice point $m\neq\0$ in the interior of $Q/(k-1)$. This lattice point is contained in the cone $\pos{F}:=\R_{\geq 0} F$ over some facet $F$ of $Q$. Therefore, $1\leq\pro{F}{m}<\locidx_F/(k-1)$. This implies that $k-1<\locidx_F$, and thus $k \leq\locidx_F$.

It follows from a more general result of Lagarias and Ziegler~\cite{LZ91} that, up to unimodular equivalence, there are only finitely many IP-polygons $Q$ of order $\order_Q \leq k$, for any positive integer $k$. Note that we do not yet know a sharp bound on the maximal volume in terms of the order (cf.~\cite{Pik01}), although there exist candidates (see Example~\ref{eg:triangles}).

In Section~\ref{sec:volume_bound} we provide asymptotically sharp upper bounds in terms of the maximal local index. It is unknown whether these bounds are also asymptotically optimal for LDP-polygons. Theorem~\ref{thm:bound_boundary} and Corollary~\ref{cor:bound_volume} are summarised in the following statement:
\begin{theorem}\label{thm:summary_bounds}
Let $Q$ be an IP-polygon of maximal local index $\maxlocidx_Q = k \geq 2$. Then:
\[\card{\rand Q \cap N}\leq4k^2+8k,\quad\Vol{Q}\leq4k^3+8k^2.\]
\end{theorem}

As implied by the terminology, when $Q$ is an LDP-polygon the index $l_Q$ of $Q$ equals the index $l$ of $X_Q$. The \emph{dual polygon} $Q^\vee$ is defined as:
\[Q^\vee:=\{u\in M_\Q\st\spro{u}{v}\geq -1\text{ for all }v\in Q\}.\]
$Q^\vee$ is a polygon containing the origin in its interior, with:
\[\V{Q^\vee}=\{-\eta_F/\locidx_F\st F\in\F{Q}\}.\]
Hence the index $\idx_Q$ equals the smallest positive integer $k$ such that $k Q^\vee$ is a lattice polygon; i.e., the smallest positive multiple such that $-k K_Q$ is a Cartier divisor. 

It is well-known that:
\[\order_Q=1\iff\maxlocidx_Q=1\iff\idx_Q=1.\]
Such polygons are called \emph{reflexive} (and the corresponding varieties \emph{Gorenstein}). There are exactly sixteen reflexive polygons, of which five are triangles. It is worth observing that the definitions generalise to higher dimensions; reflexive polytopes have been classified up to dimension four~\cite{KS98b,KS00} and are of particular relevance to the study of Calabi-Yau hypersurfaces~\cite{Bat94}. It is conjectured that their maximal volume in fixed dimension is the same as the maximal volume for IP-polygons of order one, however effective bounds are still open (see~\cite{Nill04}).

In Section~\ref{sec:index_2_classify} we classify all LDP-polygons with $l_Q=2$: there are thirty cases. 
Of these, seven are known to be triangles~\cite[Theorem~6.12]{Dais06}; this should be contrasted with the non-toric results of~\cite{Koj03}. Dais has also classified all LDP-triangles with index three~\cite{Dais07}, yielding eighteen cases.

In Sections~\ref{sec:first_classifiction_alg} and~\ref{sec:second_classification_alg} we present two independent methods for classifying all LDP-polygons. The first is inductive on the maximum local index $\maxlocidx_Q$ and uses Theorem~\ref{thm:summary_bounds}. The second fixes the index $\idx_Q$ and employs the concept of special facets introduced in~\cite{Obr07}. A computer algorithm has been implemented which has classified all LDP-polygons with $\idx_Q\leq16$. The resulting classifications can be obtained via the Graded Rings Database~\cite{GRDb} at \href{http://malham.kent.ac.uk/}{\texttt{http://malham.kent.ac.uk/}} and are summarised below.

\begin{theorem}\label{thm:classification_data}
For each positive integer $k$ let $n(k)$ be the number of isomorphism classes of toric log del Pezzo surfaces with index $\idx=k$, and let $m(k)$ be the number of rank one toric log del Pezzo surfaces with index $\idx=k$. Then:

\[
\begin{array}{|r|c|c|c|c|c|c|c|c|}
\hline
k&1&2&3&4&5&6&7&8\\\hline
n(k)&16&30&99&91&250&379&429&307\\\hline
m(k)&5&7&18&13&33&26&45&27\\\hline
\multicolumn{9}{c}{}\vspace{-0.6em}\\\hline
k&9&10&11&12&13&14&15&16\\\hline
n(k)&690&916&939&1279&1142&1545&4312&1030\\\hline
m(k)&51&51&67&53&69&74&133&48\\\hline
\end{array}
\]
\end{theorem}

\begin{acknowledge}
The first author would like to express his gratitude to Colin Ingalls for several useful discussions whilst at the University of New Brunswick. He is currently supported by EPSRC grant EP/E000258/1. The second author is supported in part by the Austrian Research Funds FWF under grant number P18679-N16. The third author is a member of the Research Group Lattice Polytopes supported by Emmy Noether Fellowship HA 4383/1 of the DFG. We would like to thank Dimitrios Dais for initiating this research.
\end{acknowledge}

\section{The projection method}\label{sec:projection_method}
In this section we explain an elementary observation used in Section~\ref{sec:index_2_classify} to give a direct classification of all LDP-polygons of index two.

First we require a variant of the projection property of reflexive polytopes (see~\cite[Proposition~4.1]{Nill05}):

\begin{lemma}\label{lem:projection}
Let $Q$ be an LDP-polygon and let $F \in \F{Q}$ be a facet with $\locidx_F =\maxlocidx_Q$. Assume there exists a non-vertex lattice point $x \in \pos{F}$ with $\pro{F}{x} = 1$. If $m$ is a lattice point in $Q$ with $m \notin F$ then $m+x$ is also a lattice point in $Q$.
\end{lemma}
\begin{proof}
We may assume that $m+x\neq\0$. Hence there exists a facet $G\in\F{Q}$ with $m+x\in\pos{G}$. We have to show that $\pro{G}{m+x}\leq\locidx_G$. If $G=F$ then $\pro{F}{m+x}=\pro{F}{m}+1\leq\locidx_F$, since $\pro{F}{m}\leq\locidx_F-1$ by assumption. Therefore $G\neq F$ and it suffices to show that $\pro{G}{x}\leq\0$. 

Assume that $\pro{G}{x}\geq 1$. Since $\locidx_F x\in F\subseteq Q$ we see that $\locidx_F\leq\pro{G}{\locidx_F x}\leq\locidx_G$. But $\locidx_F$ was chosen to be maximal, so $\locidx_F =\locidx_G$. Hence $\locidx_F x\in F\cap G$, and so $\locidx_F x$ is a vertex (in particular, a primitive lattice point). This implies  that $\locidx_F=1$; a contradiction.
\end{proof}

\begin{figure}[htbp]
\centering
\includegraphics[scale=0.8]{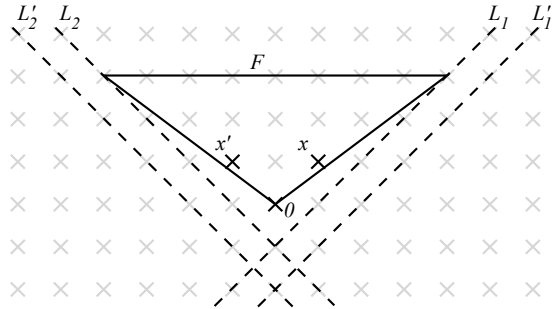}
\caption{Illustration of the proof of Proposition~\ref{prop:facetbound}.}
\label{fig:proof_of_facetbound}
\end{figure}

Here is our main application (the proof is illustrated in Figure~\ref{fig:proof_of_facetbound}):

\begin{proposition}\label{prop:facetbound}
Let $Q$ be an LDP-polygon with maximal local index $\maxlocidx_Q=k\geq 2$, and suppose that $F\in\F{Q}$ has local index $\locidx_F=k$. Then:
\[\card{F\cap N}\leq 2k^2.\]
\end{proposition}
\begin{proof}
Assume that $b:=\card{F\cap N}-1\geq 2k^2$, so $F/k$ has lattice length $b/k\geq 2k$. By a unimodular transformation we may assume that there exists $x:=(0,1)\in N\cap F/k$ and $x':=(g-1,1)\in N\cap F/k$, for some $g\in\Z_{\geq 1}$, 
such that $(-1,1)\notin F/k$ and $(g,1)\notin F/k$. This implies that $g\geq\lfloor b/k\rfloor\geq 2 k$.

Let the vertices of $F$ be $(f,k)$ and $(e,k)$, where $f\leq 0\leq e$. We know that $b=e-f$. Let $L_1$ be the line through $(f,k)$ with direction vector $-x$ and let $L_2$ be the line through $(e,k)$ with direction vector $-x'$. These intersect at the point $p:=(f,k-b/(g-1))$. Lemma~\ref{lem:projection} applied to $x$ and $x'$ yields that $Q$ is contained in the triangle $\sconv{(f,k),(e,k),p}$.

Since $(-k,k)\notin F$ and $(kg,k)\notin F$ by assumption, $-k+1\leq f$ and $e\leq kg-1$. Let $L'_1$ and $L'_2$ be the lines through $(-k+1,k)$ with direction vector $-x$, and through $(kg-1,k)$ with direction vector $-x'$, respectively. They intersect in the point $p':=(-k+1,k-(kg+k-2)/(g-1))$. $Q$ must be contained in the triangle $\sconv{(-k+1,k),(kg-1,k),p'}$.

This implies that $k-(kg+k-2)/(g-1)\leq -1$, since $\0$ is in the interior of $Q$, yielding that $kg-k-kg-k+2\leq -g+1$. Thus $g+1\leq 2k$; a contradiction.
\end{proof}

\section{The classification of LDP-polygons of index two}\label{sec:index_2_classify}
Using the results of the previous section we derive the following:

\begin{theorem}\label{thm:index_2_classify}
There are precisely thirty LDP-polygons of index two, up to unimodular equivalence. 
\end{theorem}
\begin{proof}
Let $Q$ be a LDP-polyon of index two. Let $F$ be a facet of $Q$ with $\locidx_F=2$, chosen such that $\card{F \cap N}$ is maximal. By a unimodular transformation we may assume that $F=\sconv{(b,2),(1,2)}$, where $b\leq-1$ is an odd integer. By Proposition~\ref{prop:facetbound} we have that $b\geq-7$. We define $\locidx_F:=\{x\in\pos{F}\cap N\st\pro{F}{x}=1\}$, so $\card{I_F}=(1-b)/2$. There are three cases to consider:
\begin{enumerate}
\item\underline{$b=-1$}.\\
\noindent Let $I_F=\{x\}$. By Lemma~\ref{lem:projection} we may assume that $Q$ lies between the two dashed lines:
\begin{center}
\vspace{0.3cm}
\includegraphics[scale=0.8]{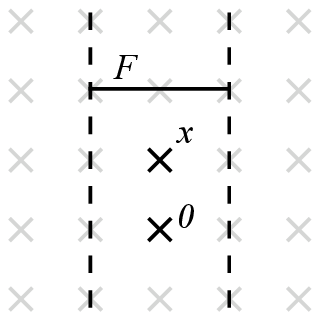}
\end{center}
There are three possibilities:
\begin{enumerate}
\item $-x \notin P$:
\begin{center}
\vspace{0.3cm}
\includegraphics[scale=0.7]{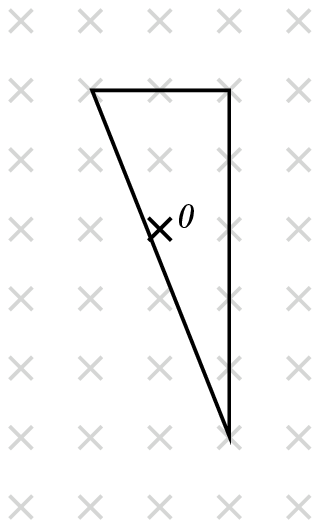}\hspace{0.3cm}
\includegraphics[scale=0.7]{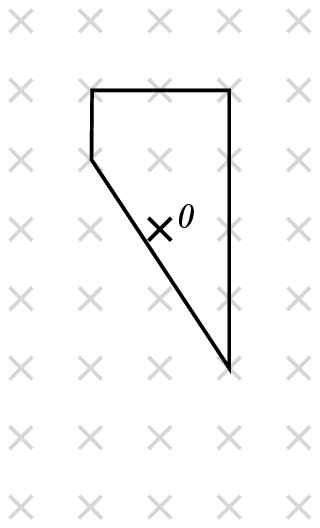}\hspace{0.3cm}
\includegraphics[scale=0.7]{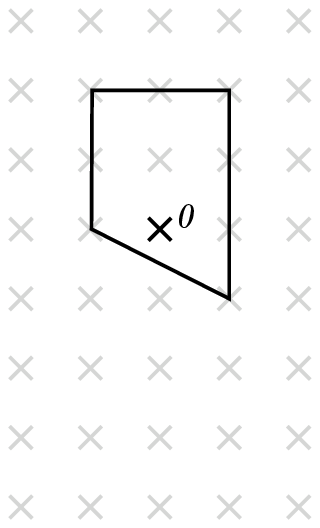}
\end{center}

\item $-x \in \rand P$:
\begin{center}
\vspace{0.3cm}
\includegraphics[scale=0.7]{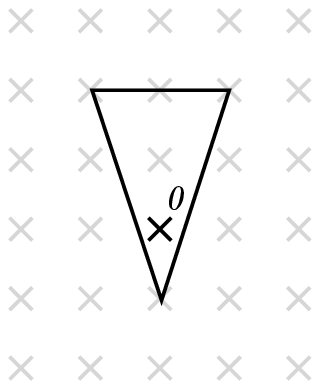}\hspace{0.3cm}
\includegraphics[scale=0.7]{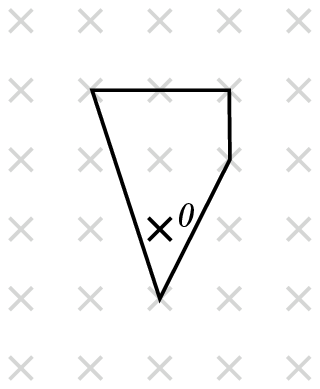}\hspace{0.3cm}
\includegraphics[scale=0.7]{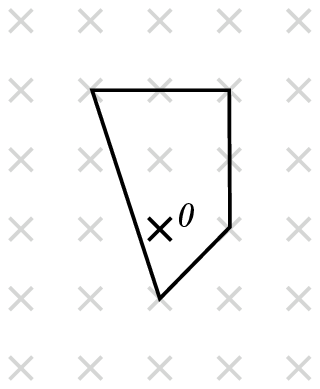}\hspace{0.3cm}
\includegraphics[scale=0.7]{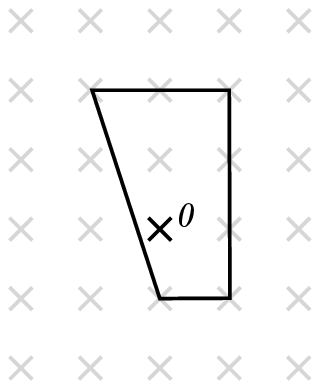}
\end{center}
\begin{center}
\vspace{0.3cm}
\includegraphics[scale=0.7]{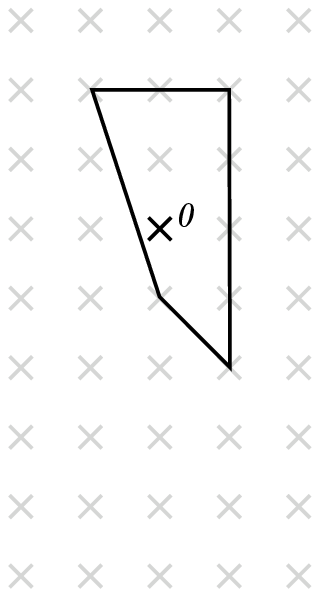}\hspace{0.3cm}
\includegraphics[scale=0.7]{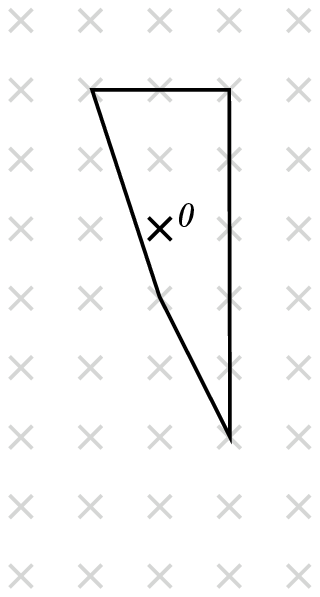}\hspace{0.3cm}
\includegraphics[scale=0.7]{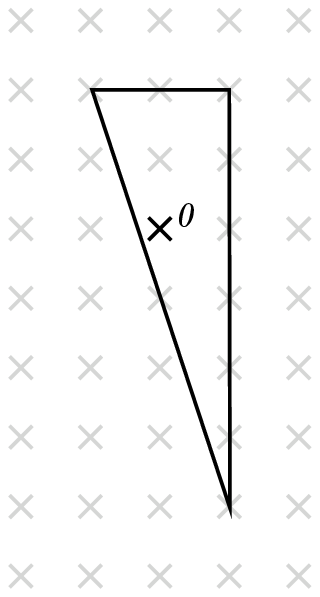}\hspace{0.3cm}
\includegraphics[scale=0.7]{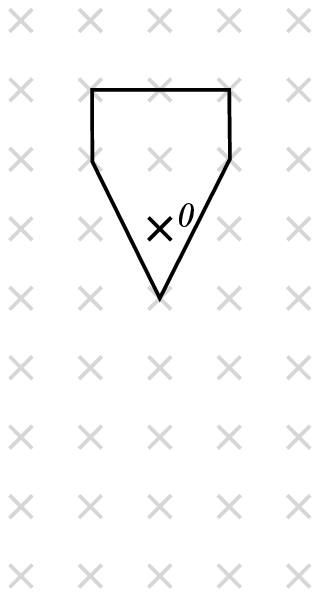}
\end{center}
\begin{center}
\vspace{0.3cm}
\includegraphics[scale=0.7]{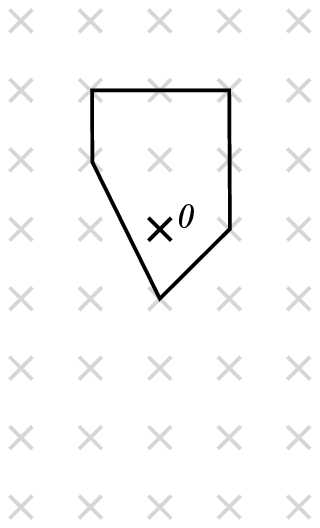}\hspace{0.3cm}
\includegraphics[scale=0.7]{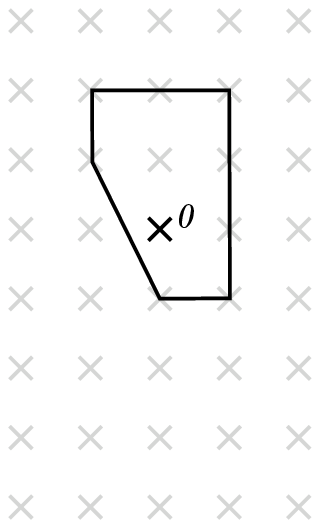}\hspace{0.3cm}
\includegraphics[scale=0.7]{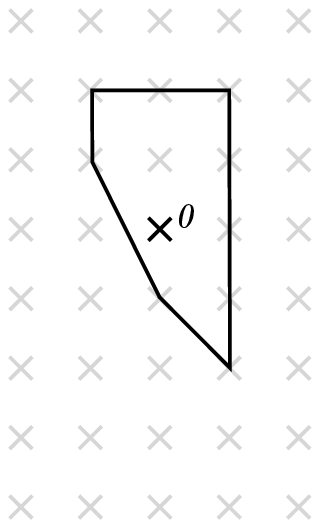}\hspace{0.3cm}
\includegraphics[scale=0.7]{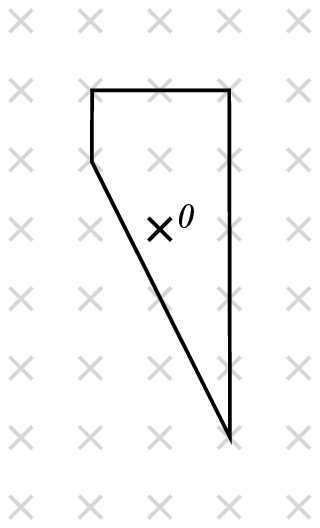}
\end{center}
\begin{center}
\vspace{0.3cm}
\includegraphics[scale=0.7]{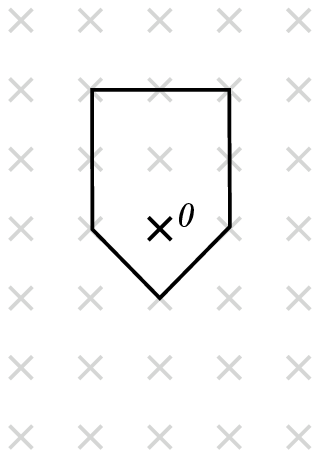}\hspace{0.3cm}
\includegraphics[scale=0.7]{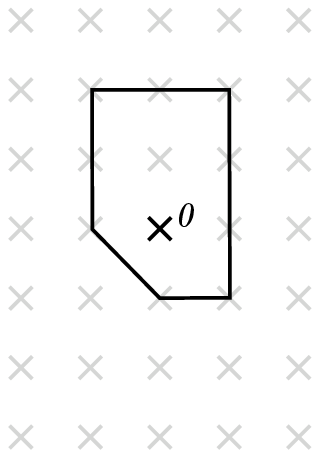}\hspace{0.3cm}
\includegraphics[scale=0.7]{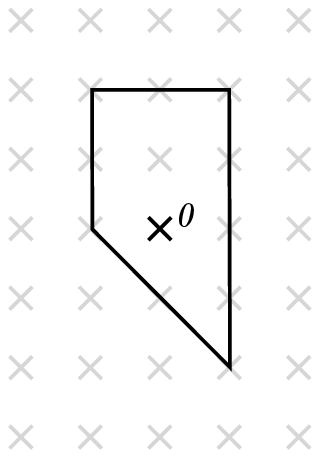}\hspace{0.3cm}
\includegraphics[scale=0.7]{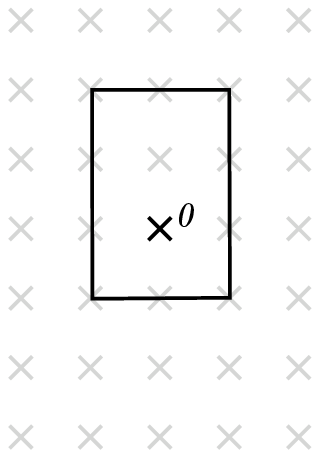}
\end{center}

\item $-x \in \intr{Q}$; i.e. $-x\in I_G$ for some facet $G$, so $-2x \in G$:
\begin{center}
\vspace{0.3cm}
\includegraphics[scale=0.7]{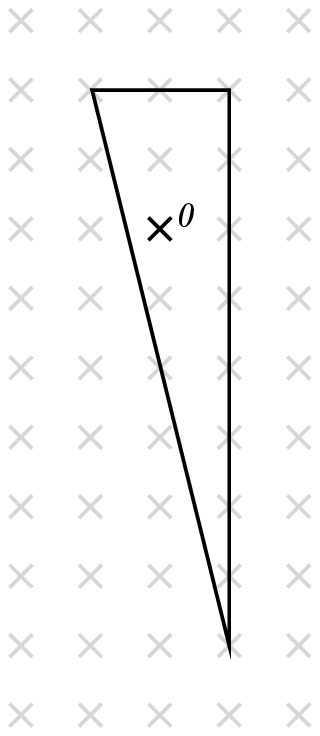}\hspace{0.3cm}
\includegraphics[scale=0.7]{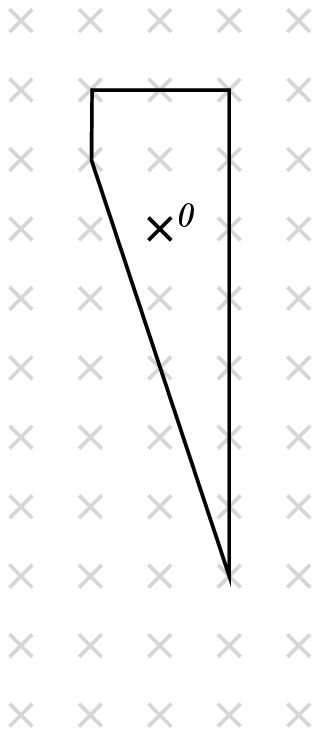}\hspace{0.3cm}
\includegraphics[scale=0.7]{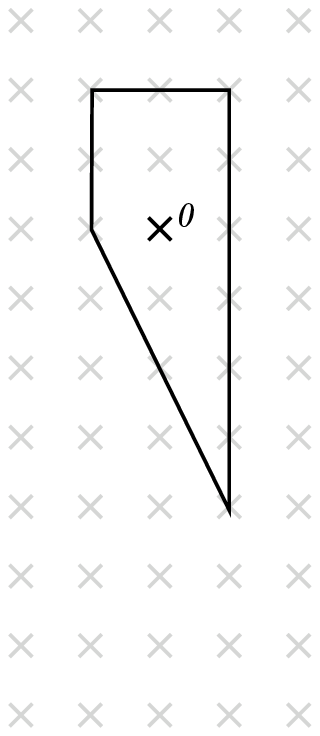}\hspace{0.3cm}
\includegraphics[scale=0.7]{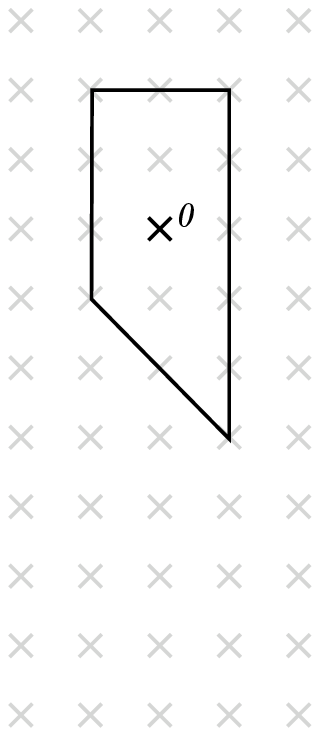}\hspace{0.3cm}
\includegraphics[scale=0.7]{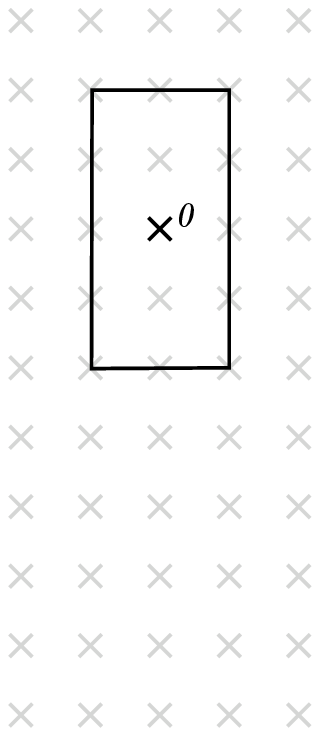}
\end{center}
\end{enumerate}
Hence we obtain $24$ LDP-polygons of index two, no pair of which are unimodularly equivalent.

\item\underline{$b=-3$}.\\
\noindent Let $I_F = \{x',x\}$, where $x$ is chosen to be to the right of $x'$. Lemma~\ref{lem:projection} implies that $Q$ lies in the region defined by the dashed lines:
\begin{center}\vspace{0.3cm}
\vspace{0.3cm}
\includegraphics[scale=0.8]{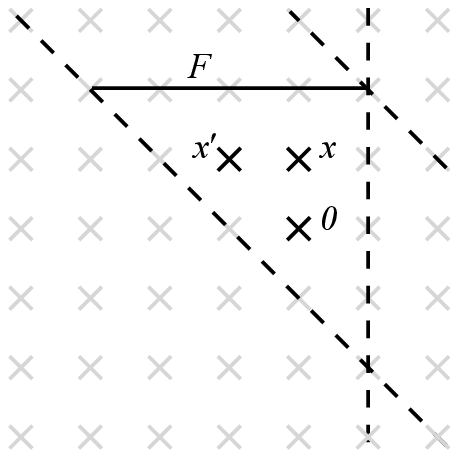}
\end{center}
Now a simple enumeration yields the following list:
\begin{center}
\vspace{0.3cm}
\includegraphics[scale=0.7]{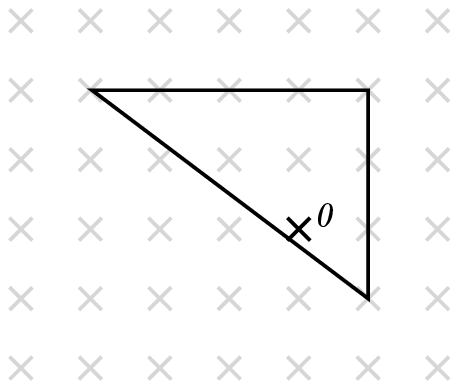}\hspace{0.3cm}
\includegraphics[scale=0.7]{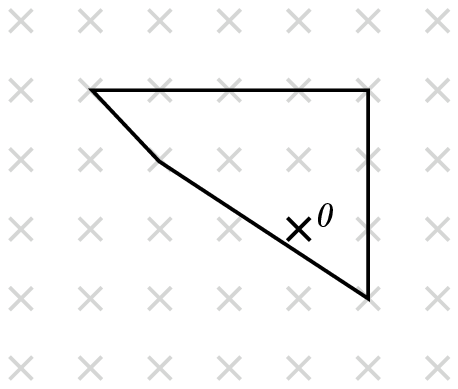}\hspace{0.3cm}
\includegraphics[scale=0.7]{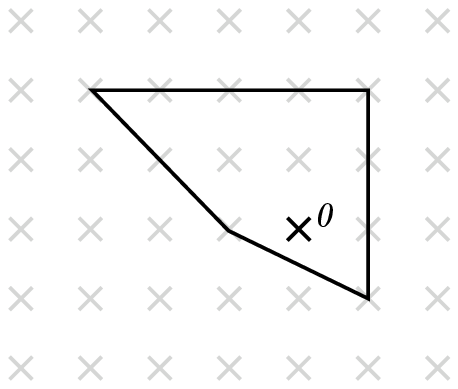}\hspace{0.3cm}
\includegraphics[scale=0.7]{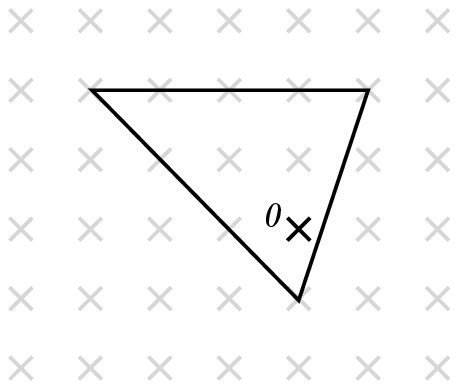}
\end{center}
\begin{center}
\vspace{0.3cm}
\includegraphics[scale=0.7]{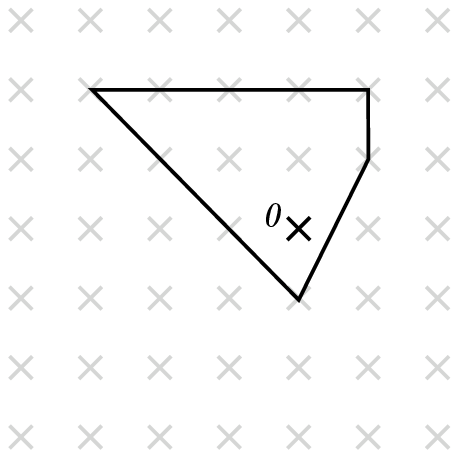}\hspace{0.3cm}
\includegraphics[scale=0.7]{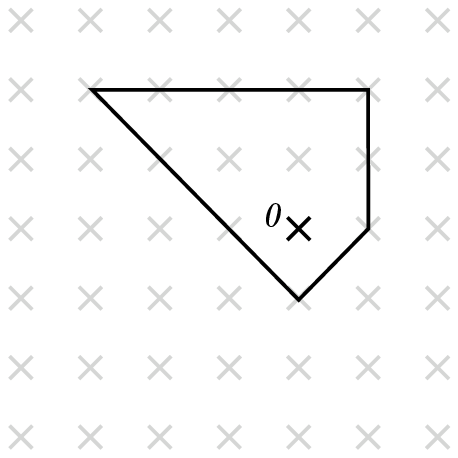}\hspace{0.3cm}
\includegraphics[scale=0.7]{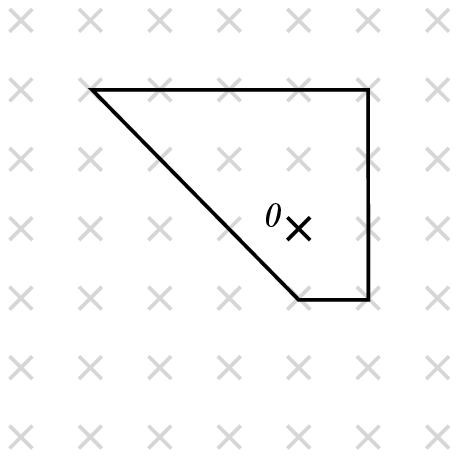}\hspace{0.3cm}
\includegraphics[scale=0.7]{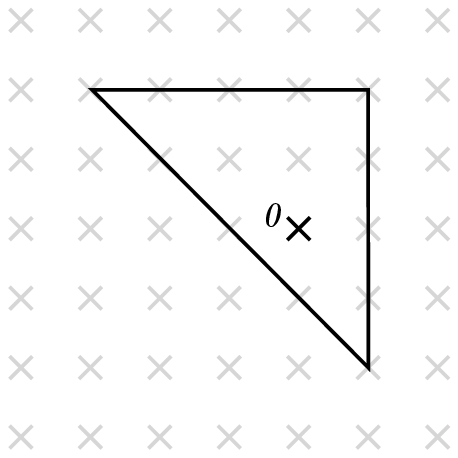}
\end{center}
Of these, the first and fourth, second and fifth, and third and sixth are unimodularly equivalent. Hence we obtain five unimodular equivalence classes.

\item\underline{$b=-5$}.\\
\noindent Let $x,x'\in I_F$, where $x$ is chosen to be the right-most lattice point in $Q$, and $x'$ to be the left-most. By Lemma~\ref{lem:projection} we have that $Q$ lies in the region enclosed by the four dashed lines:
\begin{center}\vspace{0.3cm}
\includegraphics[scale=0.8]{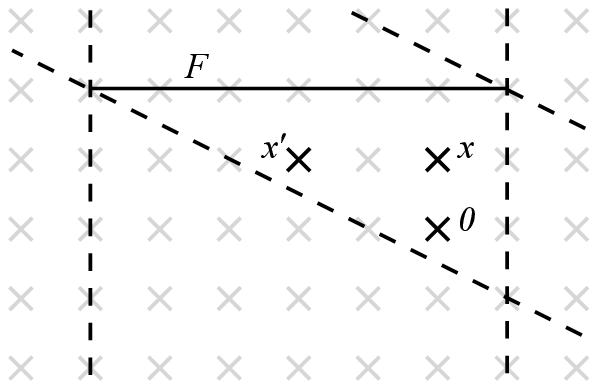}
\end{center}
This yields the following LDP-polygon, which is unique up to unimodular equivalence:
\begin{center}
\vspace{0.3cm}
\includegraphics[scale=0.7]{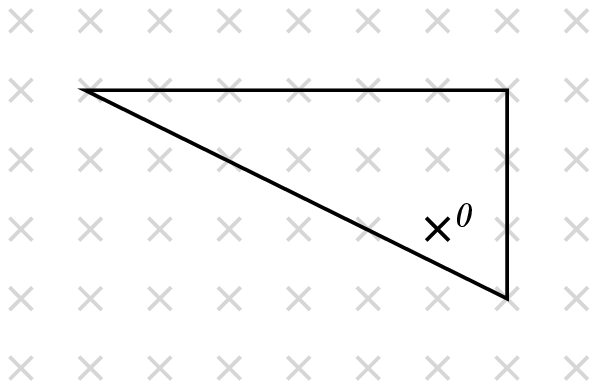}
\end{center}
\end{enumerate}
\end{proof}

\section{Bounding the volume of IP-polygons}\label{sec:volume_bound}
The main goal of this section is to present an upper bound on the volume of an LDP-polygon $Q$ of fixed maximal local index $\maxlocidx_Q = k$. In fact it is relatively easy to derive the following weak bound on the volume of an LDP-polygon in terms of the index $\idx_Q = \idx$:
\begin{equation}\label{eq:easy_bound}
\Vol{Q}\leq 8\idx^4-6\idx^3+5\idx^2.
\end{equation}
This can be seen as follows:~\cite[Theorem~1.1]{DN08} gives the quadratic bound $h\leq 8\idx^2-6\idx+5$ on the number $h$ of elements in the union of the Hilbert bases of the cones spanned by the faces of $Q$. These lattice points form a non-convex polygon $P$, where each facet has integral distance one from the origin. Therefore the volume of $P$ equals $h$. By equation~(\ref{eq:hierarchy}) $Q/\idx$ contains no non-zero interior lattice points, so $Q/\idx$ is contained in $P$. This yields~(\ref{eq:easy_bound}).

In the remainder of this section we shall generalise and improve equation~(\ref{eq:easy_bound}) by bounding the number of boundary lattice points of an IP-polygon $Q$. This suffices by the following inequality, which stems directly from the definition of the maximal local index:
\begin{equation}\label{eq:inequi}
\Vol{Q}\leq\maxlocidx_Q\card{\rand Q\cap N}.
\end{equation}
First we give a sharp upper bound on the number of lattice points in facets of IP-polygons.

\begin{proposition}\label{prop:first_IP_bound}
Let $Q$ be an IP-polygon of order $\order_Q = I$. Let $F\in\F{Q}$ be a facet with local index $\locidx_F$. Then:
\[\card{F\cap N}\leq 2I(\locidx_F+1)+1,\]
where equality implies that $\locidx_F \leq I$ and $Q$ is unimodularly equivalent to the triangle with vertices $(-I (\locidx_F+1),\locidx_F)$, $(I(\locidx_F+1),\locidx_F)$, and $(0,-1)$.
\end{proposition}
\begin{proof}
We may assume by an unimodular transformation that $F$ is the convex hull of the vertices $(a,\locidx_F), (b,\locidx_F))$ with $a \leq b$. Let $c:=b-a=\card{F\cap N}-1$ and assume that $c\geq 2I(\locidx_F+1)$. Then $(\pm I,0)\notin\intr{Q}$.

The line through $(a,\locidx_F)$ and $(-I,0)$ intersects the line through $(b,\locidx_F)$ and $(I,0)$ at a point $x$ with second coordinate $x_2=2I\locidx_F/(2I-c)$. Since $Q$ is contained in the triangle with vertices $(a,\locidx_F)$, $(b,\locidx_F)$, and $x$, and since $Q$ contains the origin in its interior, we obtain $2I\locidx_F/(2I-c)\leq-1$. This yields $c\leq 2I(\locidx_F+1)$, and hence equality. Therefore $Q$ has the vertices $(-I(\locidx_F+1),\locidx_F)$, $(I(\locidx_F+1),\locidx_F),$ and $(0,-1)$. Since $(\pm I,0)$ and $(0,\locidx_F)$ are boundary lattice points of $Q$, we see that 
$\order_Q=\max{I,\locidx_F}$. Hence, $\locidx_F\leq I$.
\end{proof}

\begin{figure}[htbp]
\centering
\includegraphics[scale=0.8]{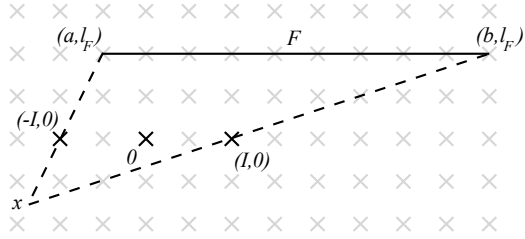}
\caption{Illustration of the proof of Proposition~\ref{prop:first_IP_bound}.}
\label{fig:proof_of_first_IP_bound}
\end{figure}



Let us consider the case of equality in Proposition~\ref{prop:first_IP_bound}.

\begin{example}\label{eg:triangles}
Let $Q=Q(a,b)$ be the triangle with the facet $F$ described in Proposition~\ref{prop:first_IP_bound} such that $a=\locidx_F$ and $b=I \geq a$. The local indices of the facets are $a$, $b$, and $b$, so $\idx_Q=\lcm{a,b}$. We compute $\Vol{Q}=2b(a+1)^2$ and $\card{\rand Q\cap N}=2(a+1)(b+1)$. 

Suppose that $a=b=\idx$, so that $\idx_Q=\idx$. In the notation of~\cite{Pik01}, $Q(\idx,\idx)$ equals the translated triangle $S_{(2,1,\idx)}$; this is conjectured to have the maximal volume of all IP-polygons of order $\idx$.
This yields a family of IP-polygons with increasing indices $\idx$, whose number of boundary lattice points grow 
as $O(\idx^2)$ and their volume grows as $O(\idx^3)$. Also note that $Q(\idx,\idx)/\idx$ (for $\idx \geq 2$) yields an unbounded family of rational triangles having only one interior lattice point and linearly increasing number of lattice points and volume.


Note that $Q$ is an LDP-polygon if and only if $\gcd{a,b}=1$. In this case, $\idx_Q=ab$. By choosing a suitable family of increasing coprime integers $a$ and $b$ we obtain a family of LDP-polygons with increasing indices $\idx$ whose number of boundary lattice points grow as $O(\idx)$ and their volume grows as $O(\idx^{3/2})$.
\end{example}

Since an LDP-polygon has primitive vertices, we obtain the following:

\begin{corollary}\label{cor:first_LDP_bound}
Let $Q$ be an LDP-polygon with maximal local index $\maxlocidx_Q\geq 2$. Then for any $F\in\F{Q}$:
\[\card{F\cap N}\leq2\maxlocidx_Q(\maxlocidx_Q+1).\]
\end{corollary}

We now present the main result of this section. The proof implicitly uses the notion of a \emph{special facet}, introduced in~\cite{Obr07}.

\begin{theorem}\label{thm:bound_boundary}
Let $Q$ be an IP-polygon of maximal local index $\maxlocidx_Q=k\geq 2$. Then:
\[\card{\rand Q\cap N}\leq 4k(k+2).\]
If $Q$ is an LDP-polygon, then:
\[\card{\rand Q\cap N}\leq 4k(k+2)-2.\]
If $Q$ is an LDP-polygon and $k\geq 3$ is prime, then:
\[\card{\rand Q\cap N}\leq 4k(k+1)-2.\]
\end{theorem}
\begin{proof}
Let:
\[p:=\sum_{x\in\rand Q\cap N}x,\]
and let $F\in\F{Q}$ be such that $p\in\pos{F}$. Hence $\pro{F}{p}\geq 0$. Set $l:=\locidx_F$, $f := \card{F\cap N}$, 
$R:=\{x\in\rand Q\cap N\st\pro{F}{x}<0\}$, and $r:=\card{R}$. We have that: 
\[-\sum_{x\in R}\pro{F}{x}\leq fl+\sum_{j=0}^{l-1}2j=l(l-1)+f l.\]
Since $f\leq 2k(k+1)+1$ and $l\leq k$ we get:
\begin{equation}\label{eq:casa}
-\sum_{x\in R}\pro{F}{x}\leq 2k^3+3k^2.
\end{equation}

Let $w:=\max{-\pro{F}{x}\st x\in R}\geq 1$. The set of lattice points on the face of $Q$ defined by $-\eta_F$ is therefore 
given by $G:=\{x\in R\st-\pro{F}{x}=w\}$. Let $g:=\card{G}\geq 1$ and $s:=r-g\geq 0$. We distinguish the case when $s$ is even and when $s$ is odd.

First suppose that $s=2t$ is even. We obtain the lower bound:
\[g(t+1)+\sum_{j=1}^t 2j\leq-\sum_{x\in R}\pro{F}{x}.\]
By equation~(\ref{eq:casa}) we are required to solve the quadratic inequality $(t+1)(t+g)\leq 2k^3+3k^2$ for $t$. This yields:
\begin{equation}\label{eq:firstrrr}
r=s+g=2t+g\leq\sqrt{8k^3+12k^2+1+g(g-2)}-1.
\end{equation}
Since $g\leq 2k(k+1)+1$ we have that:
\[\card{\rand Q\cap N}\leq f+2l+r\leq 4k^2+8k.\]

Now suppose that $s=2t+1$ is odd. We use the lower bound:
\[t+1+g(t+2)+\sum_{j=1}^t 2j\leq-\sum_{x\in R}\pro{F}{x}.\]
Proceeding as before we get:
\[r\leq\sqrt{8k^3+12k^2+g(g-4)}-1.\]
Comparing with~(\ref{eq:firstrrr}) we see that this inequality may be neglected.

Finally, if $k \geq 3$ is a prime we may use the bound $\card{F\cap N}\leq 2k^2$. If $\locidx_F = k$ this follows from Proposition~\ref{prop:facetbound}, otherwise, since $k$ is prime, we have that $\locidx_F=1$. Proposition~\ref{prop:first_IP_bound} yields $\card{F\cap N}\leq 4k+1\leq 2k^2$.
\end{proof}

Applying equation~(\ref{eq:inequi}) gives the following corollary:

\begin{corollary}\label{cor:bound_volume}
Let $Q$ be an IP-polygon of maximal local index $\maxlocidx_Q=k\geq 2$. Then:
\[\Vol{Q}\leq 4k^3+8k^2.\]
If $Q$ is an LDP-polygon, then:
\[\Vol{Q}\leq 4k^3+8k^2-2k.\]
If $Q$ is an LDP-polygon and $k\geq 3$ is prime, then:
\[\Vol{Q}\leq 4k^3+4k^2-2k.\]
\end{corollary}


\begin{remark}
The investigation of toric log del Pezzo surfaces is closely related to questions in number theory~\cite{Dais06,Dais07,DN08}. This is partially reflected by an improvement of the upper bound in Corollary~\ref{cor:bound_volume} when the index is prime, and also hinted at in Theorem~\ref{thm:classification_data} where the number of LDP-polygons appears to vary with respect to the number of distinct prime divisors in the index.
\end{remark}

When $Q$ is a \emph{centrally symmetric} IP-polygon with $\order_Q=I$, Minkowski's lattice point theorem applied to $Q/I$ yields a quadratic bound:
\begin{equation}\label{eq:mink}
\Vol{Q}\leq 8I^2.
\end{equation}

We conclude this section with some open questions. The asmpytotic order of the bounds in Theorem~\ref{thm:bound_boundary} and Corollary~\ref{cor:bound_volume} is optimal for IP-polygons, as seen from Example~\ref{eg:triangles}. Is there also an upper polynomial bound on the volume of arbitrary IP-polygons that is cubic in the \emph{order} $I$ of the polygons? The best known bound in~\cite{LZ91} is $O(I^{18})$; in~\cite{Pik01} it is claimed that in the case when $Q$ is a simplex one can show $O(I^5)$. Example~\ref{eg:triangles} tells us that $O(I^3)$ is necessary.

Considering LDP-polygons of index $\idx$, we see from Example~\ref{eg:triangles} that at least $O(\idx^{3/2})$ is required. Does there exist a family of LDP-polygons whose volume grows cubically with respect to their indices, as is the case with IP-polygons? Unfortunately we do not know the answer.

\section{Description of the first classification algorithm}\label{sec:first_classifiction_alg}
In this section we describe an algorithm to classify, up to unimodular equivalence, all LDP-polygons $Q$ with given 
maximal local index $\maxlocidx_Q \leq k$. It relies on a more general approach to compute, up to unimodular equivalence, all LDP-polygons $Q$ of given order $\order_Q \leq I$ and with $\Vol{Q} \leq V$. By equation~(\ref{eq:hierarchy}) and Corollary~\ref{cor:bound_volume} we can bound from above the order and the volume in terms of the $\maxlocidx_Q$, giving us an effective algorithm for the classification of LDP-polygons with bounded maximal local index.

Let us introduce some notation. An \emph{LDP-sub-polygon} $P$ of an LDP-polygon $Q$ is the convex hull of a subset of the vertices of $Q$ such that $P$ contains the origin in its interior. The following lemma is the basis of the algorithm.

\begin{lemma}\label{lem:classify_start}
Let $Q$ be an LDP-polygon.
\begin{enumerate}
\item\label{lem:classify_start_1} Let $T$ be an LDP-sub-triangle of $Q$. Then $T$ is unimodularly equivalent to a triangle given by the vertices $(1,0)$, $(p,q)$, and $(x,y)$, satisfying $\gcd{p,q} = 1$, $\gcd{x,y}=1$, $0\leq p<q\leq\Vol{Q}-2$, as well as $-q\leq y<0$ and $yp-q\leq xq<yp$.
\item\label{lem:classify_start_2}  Let $P$ be an LDP-sub-parallelogram of $Q$. Then $P$ is unimodularly equivalent to a parallelogram given 
by the vertices $\pm(1,0)$ and $\pm(p,q)$, where $0\leq p<q\leq 4\order_Q^2-1$. Moreover, the triangle with vertices $(1,0)$, $(p,q)$, and $(-1,-1)$ is unimodularly equivalent to an LDP-sub-triangle of $Q$.
\end{enumerate}
\end{lemma}
\begin{proof}
(1) Any LDP-sub-triangle $T$ can be decomposed into three triangles with apex $\0$ by intersecting $T$ with the cones over the three faces. By a unimodular transformation we may assume that $T=\sconv{v_1,v_2,v_3}$, 
where $C:=\sconv{\0,v_1,v_2}$ has the maximum volume of the three triangles. Since the vertices are primitive, we may assume that $v_1=(1,0)$ and that $v_2=(p,q)$, where $0\leq p<q=\Vol{C}$. Since $\Vol{C}+2\leq\Vol{Q}$, we get that $q\leq\Vol{Q}-2$. Since the volume of $C$ is maximal, we immediately see that $v_3=(x,y)$ has to be contained 
in the parallelogram $\sconv{\0,-v_1,-v_2,-v_1-v_2}$. This yields the restrictions on $x$ and $y$.

(2) This follows as before, however we use central symmetry and the bound~(\ref{eq:mink}), since $\order_P \leq \order_Q$ by definition.
\end{proof}

Let us recall an upper bound on the number of vertices:

\begin{lemma}[\protect{\cite[Lemma~3.1]{DN08}}]\label{lem:ecken}
Let $Q$ be an LDP-polygon with maximal local index $\maxlocidx_Q\geq 2$. Then:
\[\card{\V{Q}}\leq 4\maxlocidx_Q+1.\]
\end{lemma}

Using these two lemmas we can describe the four steps of the algorithm.

\algorithm{Classification of LDP-polygons $Q$ with $\order_Q\leq I$ and $\Vol{Q}\leq V$.}
\begin{enumerate}
\item \emph{Classification of all possible LDP-sub-triangles $T$:}
According to Lemma~\ref{lem:classify_start}~(\ref{lem:classify_start_1}) we proceed as follows: first we list all (finitely many) possible $(p,q)$ (with $q \leq V-2$), and check whether $\sconv{(0,0),(1/I,0),(p/I,q/I)}$ contains no interior lattice points. Then for any such $(p,q)$ we list all (finitely many) possible $(x,y)$. Finally, for each such $T$ (with $\Vol{T} \leq V$) we check whether $T/I$ contains non-zero interior lattice points.

\item \emph{Classification of all possible LDP-sub-parallelograms $P$:}
According to Lemma~\ref{lem:classify_start}~(\ref{lem:classify_start_2}) we go through the list of possible $(p,q)$ obtained in Step 1, with the additional condition $q \leq 4I^2-1$, and check that the parallelogram $P = \sconv{\pm (1/I,0), \pm (p/I,q/I)}$ has no non-zero interior lattice points, and that $\Vol{P}\leq V$.

\item \emph{Successively choosing new vertices:}
Assume that we have already constructed all possible LDP-sub-polygons with at most $t-1$ vertices. We start with $t=3$ in Step 1, and finish if $t=4I+1$ by Lemma~\ref{lem:ecken}. So let $t\geq 4$. Since in Step 2 we have already classified all LDP-sub-parallelograms, we may assume that we can obtain an LDP-sub-polygon $Q'$ with $t$ vertices recursively from an LDP-sub-polygon $Q''$ with $t-1$ vertices by adding a vertex $w$. Here we have that either $w=-v$ for a vertex $v$ of $Q''$, or there exists an edge $\sconv{v_1,v_2}$ of $Q''$ such that $w,v_1,v_2$ forms an LDP-sub-triangle. This gives only finitely many possibilities for the choice of the new vertex $w$, according to the list in Step 1. Of course it is useful  to immediately impose, for each new selection of a vertex, convexity of the resulting polytope.

\item \emph{Identifying unimodular equivalence:}
The redundancy of the construction can be reduced by starting with a maximal LDP-sub-triangle with respect to some fixed total ordering, i.e. by using only triangles that are smaller or equal to the initial triangle during the refinement process. The remaining redundancy of representatives of $GL(2,\Z)$ equivalence classes can be addressed, for example, by bringing the polygons to a normal form using PALP~\cite{KS04}. 
\end{enumerate}

\section{Description of the second classification algorithm}\label{sec:second_classification_alg}
The aim of this section is to describe an algorithm for classifying all LDP-polygons with index $\idx_Q=\idx$, for some fixed positive integer $\idx$. This approach stem from an ingenious definition by {\O}bro: that of the \emph{special facet}, put to impressive use in~\cite{Obr07}. A facet $F$ of an IP-polygon $Q$ is said to be \emph{special} if:
\[\sum_{v\in\mathcal{V}(Q)}v\in\pos{F}.\]
Clearly when $0\in\intr{Q}$ there always exists at least one special facet.

\begin{lemma}[\protect{c.f.~\cite[Lemma~3.1]{DN08}}]\label{lem:vertical_bound}
Let $Q$ be an LDP-polygon and let $F$ be a special facet of $Q$ with local index $\locidx_F$. Then:
\[Q\subset\{v\in N_\Q\st-\locidx_F(\locidx_F+1)\leq\pro{F}{v}\leq\locidx_F\}.\]
\end{lemma}
\begin{proof}
We partition the vertices of $Q$ into two sets:
\begin{align*}
\mathcal{V}_{<0}(Q)&:=\{v\in\V{Q}\st\pro{F}{v}<0\},\text{ and }\\
\mathcal{V}_{\geq 0}(Q)&:=\{v\in\V{Q}\st\pro{F}{v}\geq 0\}.
\end{align*}
Since $F$ is a facet, for each vertex $v$ of $Q$ there exists some integer $k\leq\locidx_F$ such that:
\begin{equation}\label{eq:vertex_height}
\pro{F}{v}=k.
\end{equation}
Furthermore, since $Q$ is two dimensional, for any such $k$ there exist at most two vertices satisfying~(\ref{eq:vertex_height}). In particular we obtain:
\[\sum_{v\in\mathcal{V}_{\geq 0}(Q)}\pro{F}{v}\leq\locidx_F(\locidx_F+1).\]
Since $F$ is a special facet, we have that:
\[0\leq \pro{F}{\sum_{v\in\mathcal{V}(Q)}v}=\sum_{v\in\mathcal{V}_{<0}(Q)}\pro{F}{v}+\sum_{v\in\mathcal{V}_{\geq 0}(Q)}\pro{F}{v}.\]
Hence:
\[\sum_{v\in\mathcal{V}_{<0}(Q)}\pro{F}{v}\geq-\locidx_F(\locidx_F+1).\]
\end{proof}

The following corollary is little more than an application of Proposition~\ref{prop:first_IP_bound} (we refer the reader to Figure~\ref{fig:proof_of_first_IP_bound}):

\begin{corollary}\label{cor:bound_length}
Let $Q$ be an LDP-polygon with index $\idx_Q$, and let $F\in\F{Q}$ be a facet with local index $\locidx_F$. Then:
\[\card{F\cap N}\leq2\idx_Q(\locidx_F+1)+1.\]
In particular if we write $F=\sconv{(a,\locidx_F),(b,\locidx_F)}$, where $-\locidx_F<a\leq 0<b$, then:
\[Q\subset\{(x,y)\in N_\Q\st-\locidx_Fx+(\idx_Q+a)y\leq\locidx_F\idx_Q\text{ and }\locidx_Fx+(\idx_Q-b)y\leq\locidx_F\idx_Q\}.\]
\end{corollary}

Lemma~\ref{lem:vertical_bound} and Corollary~\ref{cor:bound_length} provide all the information required to classify the LDP-polygons with fixed index. The algorithm first fixes a special facet, and then attempts to complete that facet to an LDP-polygon via successive addition of vertices. This process is repeated for all possible choices of vertices and all possible initial special facets. A computer implementation of this algorithm was used to produce Theorem~\ref{thm:classification_data}.

\algorithm{Classification of LDP-polygons $Q$ with $\idx_Q=\idx$.}
\begin{enumerate}
\item \emph{Choosing a special facet:}
We begin by fixing a special facet $F$. This facet must have $\locidx_F\mid\idx$, and without loss of generality we may assume that $F=\sconv{(a,\locidx_F),(b,\locidx_F)}$, where $-\locidx_F<a\leq 0<b$ and (by Corollary~\ref{cor:bound_length}) $b-a\leq2\idx(\locidx_F+1)$. Furthermore, $\gcd{a,\locidx_F}=\gcd{b,\locidx_F}=1$ since vertices are primitive by definition.

\item \emph{Choosing a new vertex:}
Inductively, suppose that we have a collection of vertices $v_0$, $\ldots,$ $v_{k-1}$, where $v_0:=(a,\locidx_F)$ and $v_1:=(b,\locidx_F)$, such that $F_i:=\sconv{v_{i-1},v_i}$ is a facet for $i=1,\ldots,k-1$. We wish to find a possible choice for $v_k:=(m_k,n_k)$, $\gcd{m_k,n_k}=1$, such that $\sconv{v_{k-1},v_k}$ is a facet.

First consider the possibility $v_k=v_0$. If the resulting polygon is an LDP-polygon with index $\idx_Q=\idx$ and $F$ is a special facet, then add it to the list.

The remaining choices for $v_k$ are strictly finite. Lemma~\ref{lem:vertical_bound} places a lower bound on $n_k$; we have that:
\begin{equation}\label{eq:bounding_w}
-\locidx_F(\locidx_F+1)\leq n_k<\locidx_F.
\end{equation}
Also, by Corollary±\ref{cor:bound_length}, $v_k$ must satisfy:
\[-\locidx_Fm_k+(\idx+a)n_k\leq \locidx_F\idx\qquad\text{ and }\qquad\locidx_Fm_k+(\idx-b)n_k\leq \locidx_F\idx.\]

We can reduce this (already finite) region still further. Since $0\in\intr{Q}$ it must be that $v_k$ lies strictly on the opposite side of the hyperplane $\sspan{v_{k-1}}$ to $v_{k-2}$. It must also lie strictly within the region bounded by the existing facets: $\pro{F_i}{v_k}<\locidx_{F_i}$ for all $i=0,\ldots,k-1$. Finally, the proposed facet $F_k$ must have $\locidx_{F_k}\mid\idx$.

By considering the proof of Lemma~\ref{lem:vertical_bound} we see that the lower bound  $n_k\geq-\locidx_F(\locidx_F+1)$ of equation~(\ref{eq:bounding_w}) can be improved upon at each step of the induction. Let $b_1=-\locidx_F(\locidx_F+1)$ be the base case, and replace equation~(\ref{eq:bounding_w}) with $b_{k-1}\leq n_k<\locidx_F$. Upon choosing a new vertex $v_k$ we can write down an improved bound $b_k$ as follows:
\begin{align*}
b_k:=b_{k-1}-&\min{0,n_k}+\sum_{j\in S}j,\\
&\text{ where }S:=\{i\in\Z_{>0}\st\min{n_{k-1},n_k}<i<\max{n_{k-1},n_k}\}.
\end{align*}
Should $b_k$ become positive then $F$ cannot be a special facet for any resulting polygon: a different choice of vertices should be tried.

\item \emph{Identifying unimodular equivalence:}
Finally, the issue of unimodular equivalence should be raised. The list obtained from performing the above algorithm should be checked for equivalent polygons; i.e. one need simply check for a transformation in $GL(2,\Z)$ sending the vertices of one LDP-polygon to the vertices of another. Obviously the list of possible candidates can be substantially reduced if data such as the $\locidx_{F_i}$ and $\card{Q\cap N}$ are considered.
\end{enumerate}

\bibliographystyle{amsalpha}
\providecommand{\bysame}{\leavevmode\hbox to3em{\hrulefill}\thinspace}
\providecommand{\MR}{\relax\ifhmode\unskip\space\fi MR }
\providecommand{\MRhref}[2]{%
  \href{http://www.ams.org/mathscinet-getitem?mr=#1}{#2}
}
\providecommand{\href}[2]{#2}

\end{document}